\documentclass[11pt,a4paper]{article}
\usepackage[a4paper,nohead,left=3cm,body={15cm,21cm}]{geometry}
\usepackage{amsmath}
\usepackage{graphicx}
\usepackage{amsfonts}
\usepackage{amssymb}
\usepackage{amsmath}
\usepackage{amsthm}
\usepackage{mathrsfs}
\usepackage{stackrel}
\usepackage{algpseudocode}

\usepackage[T1]{fontenc}
\usepackage[utf8]{inputenc}

\newcommand{\si}{\sigma}
\newcommand{\A}{\alpha}

\newcommand{\eps}{\varepsilon}

\def\naw#1{\left( #1 \right)}
\def\abs#1{\left| #1 \right|}
\def\norm#1{\left\| #1 \right\|}

\def\bignaw#1{\big( #1 \big)}
\def\bigabs#1{\big| #1 \big|}

\def\Bignaw#1{\Big( #1 \Big)}

\newcommand{\R}{\mathbb{R}}

\renewcommand{\Pr}{\mathbb{P}}
\newcommand{\II}{\mathbf{1}}
\newcommand{\E}{\mathbb{E}}

\newcommand{\Var}{\text{Var}}

\def\CE#1#2{\E\big[ #1 \big| #2 \big]}

\def\dCE#1#2{\E\left[ #1 \big| #2 \right]}

\newtheorem{theorem}{Theorem}[section]
\newtheorem{lemma}[theorem]{Lemma}

\newtheorem{prop}[theorem]{Proposition}

\theoremstyle{definition}

\newtheorem*{remark1}{Remark}

\author{Pawe{\l{}} Wolff \footnote{Institute of Mathematics. University of Warsaw. Banacha 2. 02-097 Warszawa. POLAND. Email: pwolff@mimuw.edu.pl} \thanks{Research partially supported by MNiSW Grant No. N N201 397437 and FRG grant No. DMS-0652722 from the National Science Foundation}}
\title{On randomness reduction in the Johnson-Lindenstrauss lemma}
\date{}

\begin{document}

\maketitle

\begin{abstract}
\noindent
A refinement of so-called \emph{fast Johnson-Lindenstrauss transform} (Ailon and Chazelle~\cite{ailon-chazelle}, Matou\v{s}ek~\cite{matousek}) is proposed. While it preserves the time efficiency and simplicity of implementation of the original construction, it reduces randomness used to generate the random transformation. In the analysis of the construction two auxiliary results are established which might be of independent interest: a Bernstein-type inequality for a sum of a random sample from a family of independent random variables and a normal approximation result for such a sum.
\end{abstract} 


\medskip


\medskip


\bigskip

\section{Introduction}

The Johnson-Lindenstrauss lemma~\cite{johnson-lindenstrauss} is the following fact, which might appear quite surprising at the first sight.
\begin{theorem}
Let $\eps \in (0,1)$, $\mathcal{X}$ be an $N$-point subset of $\ell_2^n$ and $d \ge C \frac{\log N}{\eps^2}$, where $C>0$ is some universal constant. Then there exists a (linear) mapping $f \colon \ell_2^n \to \ell_2^d$ such that
\begin{equation}
  \label{ineq:JL-distortion}
  \forall_{x, y \in \mathcal{X}} \quad (1-\eps) \norm{x-y}_2 \le \norm{f(x)-f(y)}_2 \le (1+\eps) \norm{x-y}_2.
\end{equation}
\end{theorem}
Despite the original, purely theoretical motivation for the Johnson-Lindenstrauss lemma, it quickly became clear that this fact is of great importance in applications, especially in designing of algorithms which process high dimensional data (see e.g.~\cite{indyk-motwani, achlioptas} and references therein). For this reason, several application-oriented variants of the above result appeared quite recently, e.g.~\cite{achlioptas, ailon-chazelle, matousek, ailon-liberty, hinrichs-vybiral, vybiral, kane-meka-nelson}.
In this paper we propose a refinement of the results of Ailon and Chazelle~\cite{ailon-chazelle} and Matou\v{s}ek~\cite{matousek}. In order to put our work into context, we briefly sketch a general idea behind classical proofs of the Johnson-Lindenstrauss lemma and comment on the papers~\cite{ailon-chazelle} and~\cite{matousek} in some more details.

Most of the known proofs of the Johnson-Lindenstrauss lemma provide the existence of the map $f$ by drawing it according to some probability distribution on the space of $d \times n$ matrices and showing that it satisfies~(\ref{ineq:JL-distortion}) with positive probability.
The original proof of the Johnson-Lindenstrauss lemma~\cite{johnson-lindenstrauss} takes the map $f$ to be an orthogonal projection onto a random $d$-dimensional subspace of $\ell_2^n$ (a random subspace means here a subspace drawn according to the normalized Haar measure on the Grassmannian $G_{d,n}$). It turns out (by means of a concentration inequality or, as originally, isoperimetry) that whenever $d \ge \frac{K}{\eps^2}$ for a given constant $K > 0$, the probability that $f$ maps any fixed vector $v \in \ell_2^n$ onto a vector of length $(1 \pm \eps) (d/n)^{1/2} \norm{v}_2$ is at least $1 - C \exp(-c K)$, where $c, C >0$ are universal constants. Taking $K$ of order $\log N$ ensures that the failure probability is less than $N^{-2}$ thus taking the union bound over ${N \choose 2}$ vectors $v = x - y$ with $x, y \in \mathcal{X}$ the probability that the map $(n/d)^{1/2} f$ fails~(\ref{ineq:JL-distortion}) is less than $1/2$. Instead of using random orthogonal projections one can use a random matrix with entries that are i.i.d. Gaussian random variables~\cite{dasgupta-gupta} or a properly normalized matrix of independent random signs~\cite{achlioptas}. In each of these cases, however, the time complexity of evaluating $f(x)$ for a single point $x \in \mathcal{X}$ is $O(n d) = O(n \log N / \eps^2)$ which is not satisfactory for many applications. Also, the amount of randomness (measured in number of random bits, i.e. unbiased coin tosses) required to generate $f$ is $O(n d)$.

Ailon and Chazelle~\cite{ailon-chazelle} proposed a construction of a random map $f$ which they called a \emph{fast Johnson-Lindenstrauss transform} for the reason that in a wide range of parameters (basically $N$ vs. $n$) it is computationally more efficient than the constructions described above.  Assuming $n$ is a power of 2, they take $f = P H D$ where $D$ is an $n \times n$ diagonal matrix of random signs, $H$ is the matrix of the Walsh-Hadamard transform on $\ell_2^n$ normalized by the factor $1/\sqrt{n}$ (so that $H$ is an orthogonal matrix with all entries being $\pm 1/\sqrt{n}$), and $P$ is some sparse random $d \times n$ matrix. Clearly, the transformation $H D$ is an isometry on $\ell_2^n$. Moreover, it can be shown that with probability close to $1$, $HD$ maps any fixed unit vector $u \in \ell_2^n$ onto a vector $v \in \ell_2^n$ with small $\ell_\infty$-norm. More precisely, for any constant $C > 0$, if $\norm{u}_2 = 1$ and $V = H D u$, then with probability at least $1 - N^{-C}$,
\begin{equation}
  \label{ineq:l-infty-norm-is-small}
  \norm{V}_\infty \le C' \frac{\sqrt{\log N}}{\sqrt{n}},
\end{equation}
where $C' = C'(C) > 0$ is a constant depending on $C$ only. This property is essential for the construction of the matrix $P$ which is as follows: fix $q \in (0,1]$ and set $P$ to be a matrix of independent random entries, each entry equals $0$ with probability $1-q$ and with probability $q$ is distributed $\mathcal{N}(0, 1/(dq))$. It turns out that for any fixed unit vector $v \in \ell_2^n$ satisfying $\norm{v}_\infty \le c \sqrt{q}/\sqrt{\log (N/\eps)}$, the probability that $1-\eps \le \norm{P v}_2 \le 1+\eps$ is at least $1 - N^{-C(c)}$. Together with~(\ref{ineq:l-infty-norm-is-small}) it implies that the map $f$ will work whenever  $q \ge C (\log N) \log (N / \eps) / n$. This means the expected time of applying $P$ to a single vector is $O(d q n) = O(\log^3 N / \eps^2)$ (here we assume $\log(1/\eps) = O(\log N)$). Since the transformation $H D$ can be applied in time $O(n \log n)$ using the Fast Fourier Transform over the group $(\mathbb{Z}_2)^n$, this construction beats the previous approaches in terms of time complexity whenever $\log N = o(n^{1/2})$ and $\log N = \omega(\log n)$.

Since the usage of Gaussian random variables in the matrix $P$ generally causes some extra technical problems in a practical implementation, Matou\v{s}ek~\cite{matousek} refined the result of Ailon and Chazelle replacing Gaussian r.v.'s with random signs (Bernoulli $\pm 1$ r.v.'s). Also, in both papers, a similar property for the map $f$ as a map from $\ell_2^n$ into $\ell_1^d$ was proved.

Generating the matrix $P$ described above requires roughly $n d \log_2(1/q)$ random bits. In this paper we propose a variant of construction of Alion-Chazelle and Matou\v{s}ek which save on the amount of randomness used. Instead of fully independent entries, we let $P$ to have only independent rows, and within each row we choose $k = nq$ entries at random (without replacement) in which we put a random sign, while the remaining entries are zeros. This can be done using $O(k \log n) = O(\log^2 N \log n)$ random bits per row (see Section~\ref{section:results} for details). Additionally, if we replace the independent Bernoulli $\pm 1$ random variables from the diagonal matrix $D$ with variables which are only $O(\log N)$-wise independent (see e.g.~\cite[Proposition 6.5]{alon} or \cite[Chapter 7.6, Theorem 8]{MS}), then the construction of the map $P H D$ uses only $O(\log^3 N \log n / \eps^2)$ random bits and keeps the computational efficiency and simplicity of practical implementation of the constructions from~\cite{ailon-chazelle} and~\cite{matousek}. 

The probabilistic analysis of our construction, similarly to the one done by Matou\v{s}ek in~\cite{matousek}, relies on tail estimates for sums of random variables. However, we have to deal with sums of not independent random variables, which is due to the sampling without replacement procedure used to generate sparse rows of the matrix $P$. The main tools we established to perform the analysis is a Bernstein-type inequality and the $L^1$ Berry-Esseen bound for a sum of a random sample from a family of independent random variables. Although these results are not entirely new (see the comments following Theorem~\ref{theorem:sparse-bernstein} and Theorem~\ref{theorem:stein-approximation}), we believe they still might be of some interest. Also, having potential applications of the result in mind, we provide explicit and reasonable numerical constants in estimates of parameters of our construction.

To finish the introduction, let us mention that several results in the area of efficient Johnson-Lindenstrauss embeddings appeared recently, see~\cite{kane-meka-nelson} and references therein. Although these results beat our construction in terms of amount of randomness, the methods used there are (at least in part) quite different from ours and do not seem to work in the case of embedding into $\ell_1$.

\section{Notation}
\label{section:notation}

Throughout this work
$\eps \in (0,1)$, $\delta \in (0,\frac12)$ and a positive integer $n$ are fixed parameters. Our goal is to construct a random linear map $f$ which acts from $\ell_2^n$ to a space ($\ell_2$ or $\ell_1$) of a smaller dimension and satisfies the property that for any fixed $u \in \ell_2^n$,
\[
  \Pr\bignaw{ (1-\eps) \norm{u}_2 \le \norm{f(u)} \le (1+\eps) \norm{u}_2 } \ge 1-2\delta.
\]

Assume $n$ is a power of 2 (if necessary we pad $u$ with zeros). Let $d \ge 1$ and $1 \le k \le n$ be integers to be specified later. We shall consider the following families of random variables:
\begin{itemize}
  \item $\beta_1, \beta_2, \ldots, \beta_n$ are symmetric $\pm 1$ random variables and $l$-wise independent with $l := 2\lceil \log(n/\delta) \rceil$, i.e. any $l$ r.v.'s among $\beta_1, \ldots, \beta_n$ are independent. If $l > n$ then $\beta_1, \ldots, \beta_n$ are just independent.
  \item $\eps_1, \eps_2, \ldots, \eps_n$ are independent symmetric $\pm 1$ Bernoulli random variables. The random vectors $(\eps_{i,1}, \eps_{i,2}, \ldots, \eps_{i,n})$ ($i=1, \ldots, d$) are independent copies of $(\eps_1, \eps_2, \ldots, \eps_n)$.
  \item $\xi_1, \xi_2, \ldots, \xi_n$ are $0$--$1$ random variables such that the distribution of a random set $\{ j \colon \xi_j = 1\} \subset \{1, \ldots, n\}$ is uniform over all subsets of $\{1, \ldots, n\}$ with cardinality $k$. In the other words, for any $J \subset \{1, \ldots, n\}$ with cardinality $k$, $\Pr(\{ j \colon \xi_j = 1 \}
= J) = 1/{n \choose k}$. The random vectors $(\xi_{i,1}, \xi_{i,2}, \ldots, \xi_{i,n})$ ($i=1, \ldots, d$) are independent copies of $(\xi_1, \xi_2, \ldots, \xi_n)$.
\end{itemize}
Moreover, the three families are independent among themselves, that is $\sigma(\beta_j \colon j \in \{1, \ldots, n\})$, $\sigma(\eps_j, \eps_{i,j} \colon i \in \{1, \ldots, d\}, j \in \{1, \ldots, n\})$, $\sigma(\xi_j, \xi_{i,j} \colon i \in \{1, \ldots, d\}, j \in \{1, \ldots, n\})$ are independent.


For $q \in \{1, 2\}$, define a random linear map $f_q \colon \ell_2^n \to \ell_q^d$ by
\begin{equation}
  \label{eq:f}
  f_q = \frac{1}{d^{1/q}} P H D
\end{equation}
where
\[ 
  D = 
     \begin{pmatrix} \beta_1 && 0\\ &\ddots& \\ 0&&\beta_n
      \end{pmatrix},
\qquad
  P = 
 \sqrt{\frac{n}{k}}
 \begin{pmatrix}
  \xi_{1,1} \eps_{1,1} & \cdots & \xi_{1,n} \eps_{1,n} \\
    \vdots & \ddots & \vdots \\
  \xi_{d,1} \eps_{d,1} & \cdots & \xi_{d,n} \eps_{d,n}
 \end{pmatrix}
\]
and $H = H_n$ is the normalized Walsh-Hadamard matrix of size $n \times n$, that is the orthogonal matrix defined by the recursive formula
\[ \begin{split}
  H_n &= \frac{1}{\sqrt{2}} \begin{pmatrix} H_{n/2} & H_{n/2} \\ H_{n/2} & -H_{n/2} \end{pmatrix} \quad\text{if $n > 1$,} \\[1ex]
  H_1 &= (1).
\end{split} \]

The function $\log$ stands for the natural logarithm. We write $\norm{\cdot}_q$ for the $\ell_q$ norm ($1 \le q \le \infty$).

\section{The results}\label{section:results}

The main result of this paper asserts that when $d$ and $k$ are large enough, for any fixed vector $u \in \ell_2^n$ the random transformation $f_q \colon \ell_2^n \to \ell_q^d$ ($q = 1, 2$) with probability close to 1 almost preserves the norm of $u$.
\begin{theorem}[the $\ell_2$ case]
  \label{thm:l2l2}
  Assume $n$ is a power of $2$,
  \[
    d \ge 1.55 \frac{(1+2 \eps)^2}{\eps^2} \log(3/\delta) \quad \text{and} \quad
    k \ge \max\naw{\frac{8 e}{3} \log(6d/\delta), 20e} \log(2n/\delta).
  \]
Then the random linear transformation $f_2$ as defined in~(\ref{eq:f}) satisfies
  \[
    \forall_{u \in \ell_2^n} \ \Pr\bignaw{(1+\eps)^{-1} \norm{u}_2 \le \norm{f_2(u)}_2 \le (1+\eps) \norm{u}_2} \ge 1 - 2\delta
 \]
  provided $k \le n$.
\end{theorem}
\begin{theorem}[the $\ell_1$ case]
  \label{thm:l2l1}
  Assume $n$ is a power of $2$. For any constant $\kappa \in (0,1)$,
  assume 
  \[
    d \ge \frac{\pi + \sqrt{\pi/2} \frac83 \kappa \eps}{\kappa^2 \eps^2} \log(2/\delta)
      \quad \text{and} \quad
    k \ge \max\naw{\frac{9 \pi e}{4(1-\kappa)^2 \eps^2}, 20e} \log(2n/\delta).
  \]
Then the random linear transformation $f_1$ as defined in~(\ref{eq:f}) satisfies
  \[
    \forall_{u \in \ell_2^n} \ \Pr\bignaw{(1-\eps) \norm{u}_2 \le \sqrt{\pi/2} \norm{f_1(u)}_1 \le (1+\eps) \norm{u}_2} \ge 1 - 2\delta
 \]  
  provided $k \le n$.
\end{theorem}
A typical situation in which the results are applied is the one mentioned in the introduction: having a set $\mathcal{X}$ of $N$ points in $\ell_2^n$ ($n$ is a power of 2) we want to embed it into a space of (possibly much lower) dimension $d$ with distortion $1+\eps$. To this end we sample an embedding at random as specified in Theorem~\ref{thm:l2l2} or \ref{thm:l2l1}. Assuming we want the embedding to work with probability at least $1-p$, we take $\delta = p/N^2$ and apply the union bound over ${N \choose 2}$ vectors being the differences of pairs of points from $\mathcal{X}$ to obtain that the embedding fails to have distortion $1+\eps$ with probability at most ${N \choose 2} 2\delta < p$.

Let us now discuss some implementation aspects of the embeddings.
In the case of the embedding into $\ell_2$, Theorem~\ref{thm:l2l2} implies that for any set $\mathcal{X}$ of $N$ points in $\ell_2^n$, any distortion parameter $\eps \in (0,1)$ and any failure probability $p \in (0,1)$, the probability that a random transformation $f_2$ with the parameters 
\[
  d= \left\lceil 1.55 \frac{(1+2 \eps)^2}{\eps^2} \log\naw{\frac{3N^2}{p}} \right\rceil, \quad
  k = \left\lceil \max\naw{7.25 \log\naw{\frac{6d N^2}{p}}, 55} \log\naw{\frac{2nN^2}{p}} \right\rceil
\]
embeds $\mathcal{X}$ into $\ell_2^d$ with distortion $1+\eps$ is at least $1-p$,
unless $k > n$. If indeed $k > n$, or even $k > n/3$ which means that the matrix $P$ has already poor sparsity, then one can use the construction of Achlioptas~\cite{achlioptas} instead. It provides a random embedding into $\ell_2^d$ with the target dimension $d$ similar to ours, roughly with constant $1$ instead of $1.55$. The embedding is given by a $d \times n$ matrix whose entries are independent random variables assuming values $1, 0, -1$ with respective probabilities $\frac16, \frac23, \frac16$. See~\cite{achlioptas} for details.

Therefore, in what follows, we assume $k \le n/3$. Sampling the random embedding $f_2$ is actually the matter of sampling the random variables $\beta_j$ ($j=1, \ldots, n$) and $\xi_{i,j} \eps_{i,j}$ ($i = 1,\ldots, d,  j=1, \ldots, n$). We shall describe a construction of these random variables on a sample space $\{0,1\}^r$ endowed with the uniform probability measure, thus $r$ will be the number of random bits that are used to sample $f_2$. First, due to the construction of Alon, Babai and Itai~\cite[Proposition 6.5]{alon}, the $l$-wise independent (here $l = 2 \lceil \log (N^2 n / p) \rceil$) symmetric $\pm 1$ random variables $\beta_1, \ldots, \beta_n$ can be constructed on the uniform sample space $\{0,1\}^{(\log_2 n + 1) l/2 + 1}$ and so $O\naw{(\log n) \log(N n)}$ random bits suffice. (Moreover, given an element of the sample space, their construction allows to compute the sequence $\beta_1, \ldots, \beta_n$ in time $O(l n \log n) = O\naw{n \bignaw{\log(N n)} \log n}$.) Next, for $i=1, \ldots, d$ the random variables $\xi_{i,1} \eps_{i,1}, \ldots, \xi_{i,n} \eps_{i,n}$ (which form the $i$th row of the matrix $P$) will be constructed by sampling a random subset $J \subset \{1, \ldots, n\}$ of cardinality $k$ and then sampling independently $k$ random signs. The random subset $J$ can be sampled according to the following algorithm:
\begin{algorithmic}
  \State $J \gets \emptyset$
  \While{$\#J < k$}
        \State
        using $\log_2 n$ random bits sample an index $j$ uniformly in $\{1, \ldots, n\}$
     \If{$j \notin J$}
       \State $J \gets J \cup \{j\}$
     \EndIf
  \EndWhile
\end{algorithmic}
The number of random bits used for sampling the $i$th row of $P$ is $k + T_i \log_2 n$, where $T_i$ is the number of iterations made by the {\bf while} loop. Note that $T_i$ is a sum of $k$ independent geometric random variables with subsequent success probabilities $1$, $\frac{n-1}{n}$, \ldots, $\frac{n-k+1}{n}$ (the success is sampling $j$ not yet contained in $J$). Since $k \le n/3$, a rough estimate gives $\E T_i \le \frac32 k$ and $\Var(T_i) \le \frac34 k$. Thus we can sample the matrix $P$ using $dk + T \log_2 n$ random bits, where $T = T_1 + \ldots + T_d$ and $T_i$'s are independent. Note that $\E T \le \frac32 d k$, $\Var(T) \le \frac34 d k$ and by Chebyshev's inequality, for $\lambda > 0$,
\[
  \Pr\naw{ T \ge \frac32 dk + \lambda \sqrt{\frac34 dk}} 
  \le \Pr\naw{ T \ge \E T + \lambda \sqrt{\Var(T)}} \le \lambda^{-2}.
\]
Hence with probability close to $1$, $T$ does not exceed some constant times $dk$. (Actually one can derive much stronger exponential tail estimate for $T$, but it is not essential here.) 

Overall, the whole construction uses $O\bignaw{(\log n)\log (N n) + d k \log n}$ random bits to sample $f_2$. Assuming the failure probability $p$ is fixed and $\log n = O(\log N)$ and $\log(1/\eps) = O(\log N)$, we have $d = O(\eps^{-2} \log N)$, $k = O\naw{(\log N)^2}$ and the number of random bits used is $O\naw{\eps^{-2} (\log N)^3 \log n}$. The time complexity of applying $f_2$ to a single point is $O(dk + n \log n) = O\naw{\eps^{-2} (\log N)^3 + n \log n}$, where $n \log n$ term is due to using the Fast Fourier Transform while applying the Walsh-Hadamard transform.

The construction of the embedding into $\ell_1$ is similar, since the random transformation $f_1$ has the same structure as $f_2$. According to Theorem~\ref{thm:l2l1}, for any $\kappa \in (0,1)$ and
\[
  d = \left\lceil \frac{3.15 + 3.4 \kappa \eps}{\kappa^2 \eps^2} \log\naw{\frac{2N^2}{p}} \right\rceil,
\quad
  k = \left\lceil \max\naw{\frac{19.3}{(1-\kappa^2) \eps^2}, 55} \log\naw{\frac{2nN^2}{p}} \right\rceil,
\]
the probability that $f_1$ embeds a given $N$-point subset of $\ell_2^n$ into $\ell_1^d$ with distortion $1+\eps$ is at least $p$, provided $k \le n$.
If $k \le n/3$, we can proceed with the same algorithm of sampling the matrix $P$. In such case, the total number of random bits used to sample $f_1$ is $O\bignaw{\log n (\log (N n)) + d k \log n}$ which is $O\naw{\eps^{-4} (\log N)^2 \log n}$ under the assumption that $p = \Theta(1)$, $\log n = O(\log N)$ and $\log(1/\eps) = O(\log N)$. If $k$ is between $n/3$ and $n$, then the sparsity of the matrix $P$ is poor, therefore we shall take $k = n$ and possibly increase $\kappa$ to reduce the target dimension $d$. For these new parameters we sample $f_1$; this time $P$ is just a matrix of independent random signs.
Finally, if $k > n$ for all $\kappa \in (0,1)$ then assuming $p = \Theta(1)$ we have $\log N = \Omega(\eps^2 n)$ or $\eps = O(\sqrt{\frac{\log n}{n}})$. Hence $d = \Theta(\eps^{-2} \log N) = \Omega(n)$ which means that the reduction of the dimension would be at most proportional (if any).

\section{Proofs}

The proofs of Theorems~\ref{thm:l2l2} and~\ref{thm:l2l1} consist of four steps, which we outline below:
\begin{enumerate}
  \item We show that for any unit vector $u \in \ell_2^n$, the random vector $V = H D u$ has typically the $\ell_\infty$ norm less than $C \sqrt{\log(n /\delta)}/\sqrt{n}$.
  \item If a unit vector $v \in \ell_2^n$ has small $\ell_\infty$ norm, then each coordinate of the random vector $W = (W_1, \ldots, W_d) = P v$, which is distributed as the sum $\sqrt{n/k} \sum_{j=1}^n \xi_j \eps_j v_j$, is well concentrated. More precisely, in the case of embedding into $\ell_2$ we shall show that $W_i^2$ is tightly concentrated around its mean $\E W_i^2$. In the case of embedding into $\ell_1$, we show concentration of $|W_i|$ around $\E|W_i|$. In both cases we use a version of Bernstein inequality for a sum of a random sample from a family of independent random variables (Theorem~\ref{theorem:sparse-bernstein}).
  \item In the $\ell_2$ case we note that $\E W_i^2$ (where $W = (W_1, \ldots, W_d) = Pv$) depends only on the length of $v$, and if $\norm{v}_2 =1$, then $\E W_i^2 = 1$. Since it is no longer true for $\E|W_i|$, in the $\ell_1$ case we shall use a normal approximation of the distribution of $W_i$ (Theorem~\ref{theorem:stein-approximation}) to show that $\E |W_i|$ is close to $\sqrt{2/\pi}$.
  \item If a random vector $W$ has all its coordinates well concentrated around a certain value then $\frac{1}{d} \norm{W}_2^2$ or $\frac{1}{d} \norm{W}_1$ is well concentrated.
\end{enumerate}
In the subsequent sections we elaborate on each of these steps in detail.

\subsection{Random signs and the Walsh-Hadamard transform}

Assume $u \in \ell_2^n$ is a unit vector and let $V = (V_1, \ldots, V_n) = H D u$. Since $H \colon \ell_2^n \to \ell_2^n$ is an isometry, $\norm{V}_2 = 1$ a.s. Also, $H$ has all entries $\pm 1/\sqrt{n}$ therefore $V_i = \frac{1}{\sqrt{n}} \sum_{j=1}^n \beta_j x_j,$
with $x_j = u_j$ or $x_j = -u_j$ depending on $i$, and in particular $\sum_{j=1}^n x_j^2 = 1$.
\begin{lemma}
  If $\sum_{j=1}^n x_j^2 = 1$ and $S = \sum_{j=1}^n \beta_j x_j$, then
  \[
    \Pr\naw{ |S| \ge \sqrt{2 e \log(2n/\delta)} } \le \delta/n.
  \]
\end{lemma}
\begin{proof}
Recall that $\beta_1, \ldots, \beta_n$ are $l$-independent random variables with $l = 2\lceil \log(n/\delta) \rceil$. For a sequence of (fully) independent Bernoulli random variables $\eps_j = \pm 1$
\[
  \E S^l = \E \Bignaw{\sum_{j=1}^n \eps_j x_j}^l
\]
(just expand the both sides, use linearity of expectation and note that each summand involves expectation of a product of at most $\min(l, n)$ distinct $\beta_j$'s or $\eps_j$'s). The classical Khintchine inequality states that for any $p \ge 2$,
\[
  \bigg(\E \Big| \sum \eps_j x_j \Big|^p \bigg)^{1/p} \le C_p \bigg(\sum x_j^2 \bigg)^{1/2}
\]
where $C_p$ is a constant depending on $p$ only. It follows e.g. from the classical hypercontractive estimates for Bernoulli random variables (see~\cite{bonami}) that the inequality holds with $C_p = \sqrt{p-1}$. Taking $p := l = 2 \lceil \log(n/\delta) \rceil$, we thus have
\[
  \bignaw{\E S^p}^{1/p} \le \sqrt{p-1} < \sqrt{2\log(n/\delta) + 1} < \sqrt{2\log(2n/\delta)}
\]
which combined with Chebyshev's inequality
\[
  \Pr\naw{|S| \ge \sqrt{e} \bignaw{\E |S|^p}^{1/p}} \le e^{-p/2} \le \delta/n
\]
finishes the proof.
\end{proof}
Taking the union bound over all coordinates of $V$ we immediately arrive with
\begin{prop}\label{prop:hadamard-small-ell-infty}
Let $u \in \ell_2^n$ be a unit vector and let $V = H D u$. Then
\[
  \Pr\naw{\norm{V}_\infty \ge \frac{\sqrt{2 e \log(2n/\delta)}}{\sqrt{n}}} \le \delta.
\]
\end{prop}

\subsection{Bernstein inequality for a random sample from independent r.v.'s}

In what follows, let $Y_1, Y_2, \ldots, Y_n$ be independent random variables with $\E Y_i = 0$. We assume all moments of $Y_i$'s are finite and for some constants $M>0$ and $\si_i^2 > 0$,
\begin{equation}
  \label{ineq:moments-like-exponential}
  \E |Y_i|^p \le \frac{p!}{2} \si_i^2 M^{p-2}, \quad \text{for any integer $p \ge 2$.}
\end{equation}
Put $\si^2 = \sum_{i=1}^n \si_i^2$. The theorem below is the classical inequality of Bernstein.
\begin{theorem}
  \label{theorem:classical-bernstein}
   Let $S = Y_1 + \ldots + Y_n$. Then for all $s>0$,
  \[
    \Pr(S \ge s) \le \exp\naw{- \frac{s^2}{2 \si^2 + 2 M s}}
     \qquad \text{and} \qquad
    \Pr(S \le -s) \le \exp\naw{- \frac{s^2}{2 \si^2 + 2 M s}}.
  \]
\end{theorem}
We shall also need a variant of the Bernstein inequality for a sum of a random sample of $k$ out of $n$ random variables $Y_1, \ldots, Y_n$.
\begin{theorem}
  \label{theorem:sparse-bernstein}
  Let $S = \sum_{i=1}^n \xi_i Y_i$ with $Y_i$'s
  satisfying~(\ref{ineq:moments-like-exponential}) and set
  $\si^2 = \sum_{i=1}^n \si_i^2$. Then
for all $s > 0$,
\[
  \Pr(S \ge s) \le \exp\naw{- \frac{s^2}{2 \frac{k}{n} \si^2 + 2 M s}}
   \qquad \text{and} \qquad
  \Pr(S \le -s) \le \exp\naw{- \frac{s^2}{2 \frac{k}{n} \si^2 + 2 M s}}.
\]
(Recall, $\Pr(\xi_i = 1) = k/n$.)
\end{theorem}
In the proof of Theorem~{theorem:sparse-bernstein} we will need a simple
\begin{lemma}\label{lemma:neg-correlation}
For any $A \subseteq \{1, \ldots, n\}$,
\[
  \E \naw{\prod_{i \in A} \xi_i} \le \prod_{i \in A} \E \xi_i = \naw{\frac{k}{n}}^{\# A}.
\]
\end{lemma}
\begin{proof}
  If $\# A > k$ then $\E \naw{\prod_{i \in A} \xi_i} = 0$, otherwise
\[ \begin{split}
  \E \naw{\prod_{i \in A} \xi_i} &= \Pr(\xi_i = 1 \text{ for each $i \in A$}) = \frac{{n - \# A \choose k - \# A}}{{n \choose k}} \\[1ex]
  &= \frac{k (k-1) \ldots (k-\# A + 1)}{n(n-1) \ldots (n-\# A + 1)} \le \naw{\frac{k}{n}}^{\# A}.
\end{split} \]
\end{proof}
\begin{proof}[Proof of Theorem~\ref{theorem:sparse-bernstein}]
Except for using Lemma~\ref{lemma:neg-correlation}, the proof follows a standard proof of Bernstein inequality. We present the proof below for the sake of completeness.

First, for any $i$ and $|t| < 1/M$,
\begin{equation}
  \label{ineq:est-laplace-transform-1}
  \E e^{t Y_i} = \sum_{k=0}^\infty \frac{t^k}{k!} \E Y_i^k \le
  1 + \frac{\si_i^2}{2} \sum_{k=2}^\infty |t|^k M^{k-2} \le
  1 + \frac{\si_i^2 t^2}{2(1-M |t|)}.
\end{equation}
To estimate $\E e^{t S}$, we condition on $\mathcal{F} = \si(\xi_i \colon i = 1, \ldots, n)$, use~(\ref{ineq:est-laplace-transform-1}) and Lemma~\ref{lemma:neg-correlation}:
\[ \begin{split}
  \E e^{tS} &= \E \prod_{i=1}^n \CE{e^{t \xi_i Y_i}}{\mathcal{F}} =
\E \prod_{i=1}^n (1 - \xi_i + \xi_i \E e^{t Y_i}) \\[1ex]
&\le \E \prod_{i=1}^n  \naw{1 + \xi_i \frac{\si_i^2 t^2}{2(1-M |t|)}}
\le \prod_{i=1}^n  \naw{1 + (\E \xi_i) \frac{\si_i^2 t^2}{2(1-M |t|)}} \\[1ex]
&\le \exp\naw{\frac{k}{n} \frac{\si^2 t^2}{2(1-M |t|)}}.
\end{split} \]
One obtains the inequality for the tail probability $\Pr(S \ge s)$ by taking $t = \frac{s}{\frac{k}{n} \si^2 + M s}$ and using Chebyshev's inequality. For the lower tail use $\Pr(S \le -s) = \Pr(-S \ge s)$.
\end{proof}
\begin{remark1}
  Since the random variables $\xi_1, \ldots, \xi_n$ are negatively associated (see~\cite{joag-dev-proschan}), the above result can be deduced (up to numerical constants) from a quite general comparison result of Shao~\cite{shao}. See also the paper of Hoeffding~\cite{hoeffding} for related results.
\end{remark1}
We use Theorem~\ref{theorem:sparse-bernstein} to obtain concentration for coordinates of the vector $W = P v$, i.e.
\[
  W_i = \sqrt{\frac{n}{k}} \sum_{j=1}^n v_j \xi_{i,j} \eps_{i,j} \quad \text{for $i=1, \ldots, d$.}
\]
\begin{prop}\label{prop:tail-of-Wi}
  Assume $v \in \ell_2^n$, $\norm{v}_2 = 1$, $\norm{v}_\infty \le \alpha$ and let $W = P v$. Then for $i = 1, \ldots, d$ and any $s > 0$,
\[
  \Pr(|W_i| \ge s) \le 2 \exp\naw{ - \frac{s^2}{2 + \frac{2}{3} (n/k)^{1/2} \alpha s}}.
\]
\end{prop}
\begin{proof}
Fix $i \in \{1, \ldots, d\}$ and set $Y_j = (n/k)^{1/2} \eps_{i,j} v_j$. The condition~(\ref{ineq:moments-like-exponential}) is satisfied with $\si^2_j = (n/k) v_j^2$ and $M = (n/k)^{1/2} \alpha/3$. Since $W_i = \sum_{j=1}^n \xi_{i,j} Y_j$ and $\si^2 = \sum_{j=1}^n \si^2_j = n/k$, Theorem~\ref{theorem:sparse-bernstein} provides the desired bound on $\Pr(|W_i| \ge s)$.
\end{proof}
For the sake of providing good numerical constants, beside the tail estimates established above we estimate a few first even moments of $W_i$ under the additional assumption
\begin{equation}
  \label{ineq:nka}
  \frac{n}{k} \A^2 \le r_0 := \frac{1}{10}.
\end{equation}
\begin{lemma}
  \label{lemma:a-few-first-moments}
  Under the assumptions of Proposition~\ref{prop:tail-of-Wi}, if~(\ref{ineq:nka}) holds then
\[
  \E W_i^4 \le 3.1, \quad
  \E W_i^6 \le 17, \quad
  \E W_i^8 \le 127, \quad
  \E W_i^{10} \le 1283.  
\]
\end{lemma}
\begin{proof}
  For $q=2,3,4,5$, write
  \[
    \E W_i^{2q} = \naw{\frac{n}{k}}^q \E \naw{\sum_{j=1}^n v_j \xi_j \eps_j}^{2q}
   \]
  and expand the right hand side. By the symmetry and independence of $\eps_j$'s, all the summands involving odd powers vanish. Using Lemma~\ref{lemma:neg-correlation} and the fact that for any integer $q_1 \ge 1$,
\[
  \sum_{j=1}^n v_j^{2q_1} \le \naw{\sum_{j=1}^n v_j^2} \norm{v}_\infty^{2(q_1-1)} \le \alpha^{2(q_1-1)},
\]
we obtain
  \[ \begin{split}
  \E W_i^4 &= \naw{\frac{n}{k}}^2 \E\naw{\sum_{j=1}^n v_j \xi_j \eps_j}^4 = \naw{\frac{n}{k}}^2 \naw{\sum_{j=1}^n v_j^4 \E \xi_j + 3 \sum_{j_1 \neq j_2} v_{j_1}^2 v_{j_2}^2 \E \naw{\xi_{j_1} \xi_{j_2}} }\\[1ex]
  &\le \naw{\frac{n}{k}}^2 \alpha^2 (k/n) + 3 \naw{\frac{n}{k}}^2 \naw{\sum_{j_1, j_2} v_{j_1}^2 v_{j_2}^2 } (k/n)^2 = \frac{n}{k} \alpha^2 + 3 \le r_0 + 3.
\end{split} \]
Similarly,
  \[ \begin{split}
  \E W_i^6 &\le \naw{\frac{n}{k}}^3 \Bigg( (k/n) \sum_j v_j^6 + {6 \choose 4\;  2} (k/n)^2 \sum_{j_1\neq j_2} v_{j_1}^4 v_{j_2}^2 + \frac{{6 \choose 2\;2\;2}}{3!} (k/n)^3 \sum_{\substack{j_1, j_2, j_3 \\ \text{(distinct)}}} v_{j_1}^2 v_{j_2}^2 v_{j_3}^2 \Bigg) \\[1ex]
  &\le
  (n/k)^2 \alpha^4 + 15 (n/k) \alpha^2 + 15 \le r_0^2 + 15r_0 + 15,
  \end{split} \]
and
  \[ \begin{split}
  \E W_i^8 &\le \naw{\frac{n}{k}}^4 \Bigg( (k/n) \sum_j v_j^8 + {8 \choose 6\;2} (k/n)^2 \sum_{j_1\neq j_2} v_{j_1}^6 v_{j_2}^2 + \frac{{8 \choose 4\;4}}{2!} (k/n)^2 \sum_{j_1\neq j_2} v_{j_1}^4 v_{j_2}^4 \\[1ex]
  &\qquad\qquad\quad+
\frac{{8 \choose 4\;2\;2}}{2!} (k/n)^3 \sum_{\substack{j_1, j_2, j_3 \\ \text{(distinct)}}} v_{j_1}^4 v_{j_2}^2 v_{j_3}^2 + \frac{{8 \choose 2\;2\;2\;2}}{4!} (k/n)^4 \sum_{\substack{j_1, j_2, j_3, j_4 \\ \text{(distinct)}}} v_{j_1}^2 v_{j_2}^2 v_{j_3}^2 v_{j_4}^2 \Bigg) \\[1ex]
  &\le (n/k)^3 \alpha^6 + 28 (n/k)^2 \alpha^4 + 35 (n/k)^2 \alpha^4 + 210 (n/k) \alpha^2 + 105 \\[1ex]
  &\le r_0^3+28r_0^2+35r_0^2+210r_0+105,
\end{split} \]
and
\[ \begin{split}
  \E W_i^{10} &\le r_0^4 + {10 \choose 8\;2} r_0^3 + {10 \choose 6\;4}r_0^3 + \frac{{10 \choose 6\; 2\; 2}}{2!} r_0^2 + \frac{{10 \choose 4\; 4\; 2}}{2!} r_0^2 + \frac{{10 \choose 4\; 2\; 2\; 2}}{3!} r_0 + \frac{{10 \choose 2\; 2\; 2\; 2\; 2}}{5!} \\[1ex]
  &= r_0^4 + 45r_0^3 + 210r_0^3 + 630r_0^2 + 1575r_0^2 + 3150r_0 + 945.
\end{split} \]
\end{proof}
We shall use the following simple lemma to handle the deviation of $W_i^2$ or $|W_i|$ from their means.
\begin{lemma}
\label{lemma:central-moments}
Assume $Y \ge 0$ a.s., $a \ge 0$ and $\Phi \colon \R_+ \to \R_+$ is non-decreasing. Then
\[
  \E \Phi(|Y - a|) \le \E \Phi(Y) + \Phi(a).
\]
\end{lemma}
\begin{proof}
  Note, that $\Phi(|Y - a|) \II_{\{Y \ge a\}} \le \Phi(Y)$ a.s. and $\Phi(|Y-a|) \II_{\{Y < a\}} \le \Phi(a)$ a.s. Summing up both inequalities and taking the expectation concludes the proof.
\end{proof}


\subsection{$\E \norm{P v}_2^2 = d$ and $\E \norm{P v}_1 \approx d \sqrt{2/\pi}$ by normal approximation}

Let $v \in \ell_2^n$ be a unit vector and $W = P v$. Note that $W_1, \ldots, W_d$ are independent and
\[
  \E W_i^2 = \frac{n}{k} \E \bigg( \sum_{j=1}^n v_j \xi_j \eps_j \bigg)^2
   = \frac{n}{k} \sum_j {v_j}^2 \E \xi_j = 1,
\]
hence $\E \norm{W}_2^2 = d$.

The case of $\ell_1$-norm is more complicated. In principle, $\E |W_i|$ depends on $v = (v_1, \ldots, v_n)$. However, under the assumptions of small $\ell_\infty$-norm of $v$ and $k$ large, the distribution of $W_i$ is approximately Gaussian and thus $\E |W_i|$ can be approximated by $\sqrt{2/\pi}$. To this end we establish a slightly more general result which can be regarded as $L^1$ Berry-Esseen bound for a random sample from a family of independent random variables.
\begin{theorem}\label{theorem:stein-approximation}
  Let $n \ge 2$, $Y_i$, $i=1,2,\ldots, n$ be independent random variables and independent of $(\xi_1, \ldots, \xi_n)$, satisfying $\E Y_i = 0$,
$\sum_{i=1}^n \E Y_i^2 = n/k$ and having finite third moment. Denote $X_i = \xi_i Y_i$ and $S = \sum_{i=1}^n X_i$. Then the Wasserstein distance between the distribution of $S$ and the standard normal distribution
\[
  d_W(S, G) := \stackrel[h \in \textup{Lip}(1)]{}{\sup} |\E h(S) - \E h(G)| \le 3 \sum_{i=1}^n \E |X_i|^3,
\]
where $G \sim \mathcal{N}(0,1)$ and
$\textup{Lip}(1)$ is a set of 1-Lipschitz functions on $\R$. Moreover,
\begin{equation}
  \label{ineq:normal-approximation}
  \bigabs{\E|S| - \sqrt{2/\pi}} \le \frac32 \sum_{i=1}^n \E |X_i|^3 = \frac32 \frac{k}{n} \sum_{i=1}^n \E |Y_i|^3.
\end{equation}
\end{theorem}
In the literature there exist many related results, most of them concerning a more general problem called combinatorial central limit theorem. However, the author was not able to find a result which implies~(\ref{ineq:normal-approximation}) with a reasonable numerical constant. The combinatorial central limit theorem roughly states that $S_n = \sum_{i=1}^n Y_{i, \pi(i)}$ where $(X_{i,j})_{i, j \le n}$ is a matrix of independent random variables having finite third moments and $\pi$ is a random permutation of the set $\{1, 2, \ldots, n\}$, independent from $X_{i,j}$'s, after proper normalization has the distribution close to the standard normal. Taking the matrix $(X_{i,j})$, whose first $k$ rows are independent copies of the random vector $(Y_1, \ldots, Y_n)$ and the remaining entries are zeros, boils down to the problem from Theorem~\ref{theorem:stein-approximation}.

For example, the result of Ho and Chen~\cite[Theorem 3.1]{ho-chen} on the combinatorial CLT implies an estimate similar to~(\ref{ineq:normal-approximation}) but asymptotically weaker. Bolthausen~\cite{bolthausen} proved an optimal error bound but only in the case of deterministic $(X_{ij})$'s. Recently, Chen and Fang~\cite{chen-fang} proved the combinatorial CLT in its general form with the optimal rate of normal approximation error. They bound the Kolmogorov distance, which is generally more difficult to handle in comparison to the Wasserstein distance. However, for our purposes the Wasserstein distance is better suited and moreover it is possible to obtain an estimate with a reasonable numerical constant.

As in the results on the combinatorial CLT mentioned above, we employ Stein's method.
Except for a few twists, we basically follow the argument presented in~\cite[Section 2]{intro-to-stein-method} which illustrates the usage of Stein's method in the most basic setting of sums of independent random variables.

\begin{proof}
  It is enough to consider a 1-Lipschitz $h \colon \R \to \R$ which is piecewise continuously differentiable. As in~\cite[Section 2.1]{intro-to-stein-method}, consider the differential equation
  \begin{equation}
    \label{eqn:stein-equation}
    f'(x) - xf(x) = h(x) - \E h(G)
  \end{equation}
  whose solution is given by the formula
  \begin{equation}
    \label{eqn:stein-solution}
    f(x) = e^{x^2/2} \int_{-\infty}^x (h(t) - \E h(G)) e^{-t^2/2} \, dt.
  \end{equation}
  Note that $f$ is $C^1$ and $f'$ is piecewise continuously differentiable, so for any $a, b \in \R$, $|f'(a) - f'(b)| \le |a-b| \norm{f''}_\infty$.
  It turns out~\cite[Lemma 2.3]{intro-to-stein-method} that $\norm{f}_\infty \le  2\norm{h'}_\infty \le 2$, $\norm{f'}_\infty \le 4\norm{h'}_\infty \le 4$ and $\norm{f''}_\infty \le 2\norm{h'}_\infty \le 2$.

  Substituting $S$ for $x$ in~(\ref{eqn:stein-equation}) and taking the expectation we get
  \[
    \E h(S) - \E h(G) = \E\bignaw{f'(S) - S f(S)}.
  \]
  Set $S^{(i)} = S - X_i$ and define
  \[
    K_i(t) = \E \Bignaw{X_i \naw{ \II_{\{0 \le t \le X_i\}} - \II_{\{X_i \le t < 0\}}}}.
  \]
  Note that $K_i(t) \ge 0$ for all $t \in \R$,
    \begin{equation}
    \label{eq:K}
     \int_{-\infty}^\infty K_i(t)\,dt = \E X_i^2\qquad \text{and} \qquad \int_{-\infty}^\infty |t| K_i(t)\,dt = \frac12 \E |X_i|^3.
  \end{equation}
  Since $Y_i$ is mean-zero and independent of $\sigma(S^{(i)}, \xi_i)$, we have
  \[
    \E \naw{X_i f(S^{(i)})} = \E Y_i \E \naw{\xi_i f(S^{(i)})} = 0.
  \]
  Therefore
  \[ \begin{split}
    \E \bignaw{S f(S)} &= \sum_{i=1}^n \E \bignaw{X_i f(S)} = \sum_{i=1}^n \E \naw{X_i \naw{f(S) - f(S^{(i)})}} \\[1ex]
   &= \sum_{i=1}^n \E \naw{X_i \int_0^{X_i} f'(S^{(i)} + t)} \, dt \\[1ex]
   &= \sum_{i=1}^n \int_{-\infty}^\infty \E \bignaw{f'(S^{(i)} + t) X_i \naw{ \II_{\{0 \le t \le X_i\}} - \II_{\{X_i \le t < 0\}}}} \, dt \\[1ex]
  &= \sum_{i=1}^n \int_{-\infty}^\infty \Pr(\xi_i = 1) \CE{f'(S^{(i)} + t) Y_i \naw{ \II_{\{0 \le t \le Y_i\}} - \II_{\{Y_i \le t < 0\}}}}{\xi_i = 1} \, dt \\[1ex]
  &= \sum_{i=1}^n \int_{-\infty}^\infty \frac{k}{n} \CE{f'(S^{(i)} + t)}{\xi_i=1} \E \bignaw{Y_i \naw{ \II_{\{0 \le t \le Y_i\}} - \II_{\{Y_i \le t < 0\}}}} \, dt \\[1ex]
  &= \sum_{i=1}^n \int_{-\infty}^\infty \CE{f'(S^{(i)} + t)}{\xi_i=1} K_i(t) \, dt.
 \end{split} \]
Since $\sum_{i=1}^n \E X_i^2 = 1$, we have
\[
  \E f'(S) = \sum_{i=1}^n \int_{-\infty}^\infty \E f'(S) K_i(t) \, dt.
\]
Combining the two preceding identities we obtain
\begin{equation}
  \label{ineq:fSf}
 \begin{split}
  \E \bignaw{f'(S) - S f(S)} &= \sum_{i=1}^n \int_{-\infty}^\infty \naw{E f'(S) - \CE{f'(S^{(i)} + t)}{\xi_i = 1}} K_i(t) \, dt \\[1ex]
   &= \sum_{i=1}^n \int_{-\infty}^\infty \bigg( \Pr(\xi_i=1) \CE{f'(S) - f'(S^{(i)} + t)}{\xi_i = 1} \\[1ex]
   & \qquad + \Pr(\xi_i=0) \naw{\CE{f'(S)}{\xi_i=0} - \CE{f'(S^{(i)} + t)}{\xi_i = 1}} \bigg) K_i(t) \, dt \\[1ex]
   &= \sum_{i=1}^n \int_{-\infty}^\infty \naw{\frac{k}{n} \textup{I} + \naw{1-\frac{k}{n}} \textup{II}} K_i(t) \, dt.
  \end{split}
\end{equation}
Next, estimate $|\textup{I}|$ and $|\textup{II}|$:
\[ \begin{split}
  |\textup{I}| &\le \dCE{\abs{f'(S^{(i)} + t + (X_i - t)) - f'(S^{(i)} + t)}}{\xi_i=1} \\[1ex]
      &\le \dCE{\norm{f''}_\infty |X_i - t|}{\xi_i=1} \le \norm{f''}_\infty (\E|Y_i| + |t|),
\end{split} \]
and
\[ \begin{split}
  |\textup{II}| &= \abs{\CE{f'(S^{(i)})}{\xi_i=0} - \CE{f'(S^{(i)} + t)}{\xi_i = 1}} \\[1ex]
  &\le \abs{\CE{f'(S^{(i)})}{\xi_i=0} -  \CE{f'(S^{(i)})}{\xi_i = 1}} + \norm{f''}_\infty |t| \\[1ex]
  &= | \textup{II}_1 - \textup{II}_2 | + \norm{f''}_\infty |t|.
\end{split} \]
Define a random variable $J$ which is independent of $(Y_1, \ldots, Y_n)$ and given the vector $(\xi_1, \ldots, \xi_n)$, $J$ is uniformly distributed on $\{j \colon \xi_j = 1\}$. Note that $\mathcal{L}(S^{(i)} - X_J \big| \xi_i = 0) = \mathcal{L}(S^{(i)} \big| \xi_i = 1)$ (both are the distribution of a sum of a random sample of $k-1$ random variables out of $Y_1, \ldots, Y_{i-1}, Y_{i+1}, \ldots, Y_n$), thus $\CE{f'\naw{S^{(i)} - X_J}}{\xi_i=0} = \textup{II}_2$. Hence
\[ \begin{split}
  |\textup{II}_1 - \textup{II}_2| &\le \dCE{\abs{f'\naw{\bignaw{S^{(i)} - X_J} + X_J} - f'(S^{(i)} - X_J)}}{\xi_i=0} \\[1ex]
       &\le \norm{f''}_\infty \CE{|X_J|}{\xi_i=0}.
\end{split} \]
Since $\mathcal{L}(J \big| \xi_i = 0)$ is uniform on $\{1, \ldots, n\} \setminus \{i\}$ and $\xi_J = 1$ a.s.,
$\CE{|X_J|}{\xi_i=0} = \frac{1}{n-1} \sum_{j \neq i} \E |Y_j|$
and thus
\[
  |\textup{II}| \le \norm{f''}_\infty \naw{|t| + \frac{1}{n-1} \sum_{j \neq i} \E|Y_j|}.
\]
Plugging the bound on $\textup{I}$ and $\textup{II}$ into~(\ref{ineq:fSf}) and using~(\ref{eq:K}) and the identity $\E |X_i|^p = \frac{k}{n} \E |Y_i|^p$ (for $p=1,2,3$), we obtain
\begin{align*}
  \E & \bignaw{f'(S) - S f(S)} \\[1ex]
     &\le \norm{f''}_\infty \sum_{i=1}^n \int_{-\infty}^\infty \bigg(|t| + \frac{k}{n} \E|Y_i| + \naw{1-\frac{k}{n}} \frac{\sum_{j \neq i} \E |Y_j|}{n-1}\bigg) K_i(t) \, dt \\[1ex]
 &= \norm{f''}_\infty \sum_{i=1}^n \bigg(
 \frac12 \E |X_i|^3 + \naw{\frac{k}{n}}^2 \E|Y_i| \E Y_i^2 + \naw{1 - \frac{k}{n}} \frac{\sum_{j\neq i} \E |Y_j|}{n-1} \naw{\frac{k}{n} \E Y_i^2} \bigg) \\[1ex]
&\le \norm{f''}_\infty \bigg(\frac12 \sum_{i=1}^n \E |X_i|^3 + \naw{\frac{k}{n}}^2 \sum_{i=1}^n \E|Y_i|^3 \\
&\qquad\qquad\quad + \frac{k}{n} \naw{1-\frac{k}{n}} \frac{1}{n-1} \sum_{i=1}^n \sum_{j \neq i} (\E |Y_j|^3)^{1/3} (\E |Y_i|^3)^{2/3} \bigg),
\end{align*}
where the last inequality follows from the H\"{o}lder inequality. Now, by a standard rearrangement inequality we note that for any integer $s$,
\[
  \sum_{i = 1}^n (\E|Y_{i+s}|^3)^{1/3} (\E |Y_i|^3)^{2/3} \le \sum_{i = 1}^n (\E|Y_i|^3)^{1/3} (\E |Y_i|^3)^{2/3} = \sum_{i=1}^n \E|Y_i|^3
\]
(the index $i+s$ is taken modulo $n$ and we set $Y_0 = Y_n$). Hence,
\[
\sum_{i=1}^n \sum_{j \neq i} (\E |Y_j|^3)^{1/3} (\E |Y_i|^3)^{2/3} =
\sum_{s=1}^{n-1} \sum_{i=1}^n (\E|Y_{i+s}|^3)^{1/3} (\E |Y_i|^3)^{2/3}
\le (n-1) \sum_{i=1}^n \E|Y_i|^3.
\]
Finally,
\[
  \E \bignaw{f'(S) - S f(S)} \le \frac32 \norm{f''}_\infty \sum_i \E |X_i|^3.
\]
Together with the bound $\norm{f''}_\infty \le 2\norm{h'}_\infty$ it yields the estimate for $d_W(S, G)$. To bound $\bigabs{\E|S| - \sqrt{2/\pi}}$, we take an explicit solution to the equation~(\ref{eqn:stein-equation}) for $h(x) = |x|$:
\[
  f(x) = \begin{cases}
           1 - 2e^{x^2/2} \Phi(x) & \text{for $x \le 0$,} \\[1ex]
           2 e^{x^2/2} (1 - \Phi(x)) - 1 & \text{for $x > 0$,}
         \end{cases}
\]
where $\Phi(x) = \frac{1}{\sqrt{2\pi}} \int_{-\infty}^x e^{-t^2/2} \,dt$ is the normal distribution function.

We shall prove that $\norm{f''}_\infty = 1$.
Note that $f(x)$ is an odd function, thus we compute $f''(x)$ for $x > 0$ only:
\[
  f''(x) = 2e^{x^2/2} (1-\Phi(x)) (1+x^2) - x \sqrt{2/\pi}.
\]
Since $\lim_{x \to 0_+} f''(x) = 1$, it suffices to show $f''(x) > 0$ and $f'''(x) < 0$ for all $x>0$.
To this end, we use the estimates for the Gaussian tail proved in~\cite{szarek-werner}: for all $x>-1$,
\begin{equation}
  \label{ineq:szarek-werner}
  \sqrt{2/\pi} \frac{1}{x+\sqrt{x^2+4}} \le e^{x^2/2} (1-\Phi(x)) \le \sqrt{2/\pi} \frac{2}{3x+\sqrt{x^2+8}}.
\end{equation}
We use the lower bound from~(\ref{ineq:szarek-werner}) to prove $f''>0$:
\[
  \sqrt{\pi/2} f''(x) \ge \frac{2(1+x^2)}{x+\sqrt{x^2+4}} - x = \frac{x^2 + 2 - x\sqrt{x^2+4}}{x+\sqrt{x^2+4}}
\]
and note that $(x^2+2)^2 = x^4 + 4x^2 + 4 > \naw{x\sqrt{x^2+4}}^2 = x^4 + 4x^2$.

To prove $f'''(x) = 2x(3+x^2)e^{x^2/2}(1-\Phi(x)) - (x^2+2)\sqrt{2/\pi}<0$ we use the upper bound from~(\ref{ineq:szarek-werner}):
\[
  2x(3+x^2)e^{x^2/2}(1-\Phi(x)) \le \sqrt{2/\pi} \frac{4x(3+x^2)}{3x+\sqrt{x^2+8}}.
\]
Now it suffices to prove $4x(3+x^2) < (x^2+2)(3x+\sqrt{x^2+8})$, or equivalently
\[
  x^3+6x < (x^2+2)\sqrt{x^2+8},
\]
which is obvious by calculating $\textup{LHS}^2 - \textup{RHS}^2 = -32 < 0$.
\end{proof}
Specializing Theorem~\ref{theorem:stein-approximation} to the random variables $Y_j = (n/k)^{1/2} \eps_{i,j} v_j$ we arrive with
\begin{prop}\label{prop:normal-approx}
  Assume $v \in \ell_2^n$, $\norm{v}_2 = 1$, $\norm{v}_\infty \le \alpha$ and let $W = P v$. Then for all $i = 1, \ldots, d$,
\[
  \abs{\E |W_i| - \sqrt{2/\pi}} \le \frac32 \alpha \sqrt{n/k}.
\]
\end{prop}
\begin{proof}
  By~(\ref{ineq:normal-approximation}), 
  \[
    \abs{\E |W_i| - \sqrt{2/\pi}} \le \frac32 \frac{k}{n} \sum_{j=1}^n \naw{\frac{n}{k}}^{3/2} |v_j|^3 \le \frac32 \alpha \sqrt{n/k}.
  \]
\end{proof}

\subsection{Concentration of $\norm{P v}_q$ and the proof of the main result}

Let $v \in \ell_2^n$ be a unit vector and $W = P v$. In the two next subsections we shall provide the deviation bounds for $\bigabs{\norm{W}_2^2 - 1}$ and $\abs{\norm{W}_1 - \sqrt{2/\pi}}$ as well as combine all the auxiliary results to prove Theorems~\ref{thm:l2l2} and~\ref{thm:l2l1}.

\subsubsection{The $\ell_2$ case ($q=2$)}

\begin{prop}\label{prop:deviation-for-sum-of-squares}
Assume $v \in \ell_2^n$, $\norm{v}_2 = 1$, $\norm{v}_\infty \le \alpha$ and let $W = P v$. If (\ref{ineq:nka}) holds, then for any $t \ge 0$,
  \[
    \Pr(W_1^2 + \ldots + W_d^2 \ge d + t) \le \exp\naw{-\frac{t^2}{6.2 d + 12t}} + \exp\naw{-\frac{3k}{4n \A^2} + \log(2d)}
  \]
  and
  \[
    \Pr(W_1^2 + \ldots + W_d^2 \le d - t) \le \exp\naw{-\frac{t^2}{6 d}}.
  \]
\end{prop}
\begin{proof}
  Denote $Z = W_1^2 + \ldots + W_d^2$. Recall that $W_1, \ldots, W_d$ are independent, $\E W_i^2 = 1$ thus $\E Z = d$.

  First we estimate the upper tail. Let $s_0 = \frac32 \frac{(k/n)^{1/2}}{\A}$. Proposition~\ref{prop:tail-of-Wi} gives
  \begin{equation}
    \label{ineq:for-S-2}
    \Pr\naw{|W_i| \ge s} \le 2 \exp(-s^2/3) \quad \text{for $0 \le s \le s_0$}.
  \end{equation}
Set $X_i = W_i^2 \II_{\{|W_i|\le s_0\}}$ and $\tilde{Z} = \sum_{i=1}^d X_i$.
Clearly $\E \tilde{Z} \le \E Z = d$, hence the union bound and~(\ref{ineq:for-S-2}) imply
\begin{equation}\label{ineq:aux-for-upper-tail}
  \begin{split}
    \Pr(Z - d \ge t) &\le \Pr(\tilde{Z} - d \ge t) + \Pr\naw{\exists_i \ |W_i| > s_0} \\[1ex]
      &\le \Pr(\tilde{Z} - \E \tilde{Z} \ge t) + 2 d \exp\naw{-s_0^2/3}.
  \end{split}
\end{equation}
Next, we estimate $\Pr(\tilde{Z} - \E \tilde{Z} \ge t)$ using the classical Bernstein inequality for the sum of the variables $X_i - \E X_i$. To this end, we need to verify the condition~(\ref{ineq:moments-like-exponential}). Note that
\begin{equation}\label{ineq:Xip2}
  \E |X_i - \E X_i|^2 = \text{Var}(X_i) \le \E X_i^2 \le \E W_i^4 \le 3.1
\end{equation}
where the last inequality follows from Lemma~\ref{lemma:a-few-first-moments}. To bound higher moments of $|X_i - \E X_i|$ we use
Lemma~\ref{lemma:central-moments} and the fact that $\E X_i \le \E W_i^2 = 1$
which imply
  \begin{equation}\label{ineq:higher-moments}
    \E |X_i - \E X_i|^p \le \E X_i^p + 1 \quad \text{for any $p > 0$.}
  \end{equation}
  By~(\ref{ineq:for-S-2}), $\Pr(X_i \ge t) \le 2\exp(-t/3)$ for all $t > 0$, hence for any $p > 0$,
  \begin{equation}\label{ineq:for-Xip}
    \E X_i^p = \int_0^\infty p t^{p-1} \Pr\naw{X_i > t} \,dt \le 2 \int_0^\infty p t^{p-1} \exp\naw{-t/3} \, dt = 2 \cdot 3^p \Gamma(p+1).
  \end{equation}
  However, for $p=3,4,5$ we use Lemma~\ref{lemma:a-few-first-moments} to get
  \begin{equation}\label{ineq:for-Xip345}
    \E X_i^3 \le 17, \quad
    \E X_i^4 \le 127, \quad
    \E X_i^5 \le 1283.
  \end{equation}
  Using~(\ref{ineq:Xip2}) for $p=2$, combining~(\ref{ineq:higher-moments}) with (\ref{ineq:for-Xip345}) for $p=3, 4, 5$, and combining~(\ref{ineq:higher-moments}) with (\ref{ineq:for-Xip}) for $p \ge 6$, it is a matter of elementary calculations to verify that
  \[
    \E |X_i - \E X_i|^p \le \frac{p!}{2} \si^2 M^{p-2} \quad \text{for any integer $p \ge 2$}
  \]
  with $\si_i^2 = 3.1$ and $M=6$. Now the classical Bernstein inequality from Theorem~\ref{theorem:classical-bernstein} implies
\[
  \Pr(\tilde{Z} - \E \tilde{Z} \ge t) \le \exp\naw{- \frac{t^2}{6.2 d + 12 t}}
\]
which combined with~(\ref{ineq:aux-for-upper-tail}) yields the estimate for the upper tail.

For the lower tail we use Lemma~\ref{lemma:lower-tail} (see below):
\[
\begin{split}
  \Pr(Z - d \le -t) &= \Pr\naw{\sum_{i=1}^d \naw{1-W_i^2} \ge t} \\[1ex]
       &\le \exp\naw{-\frac{t^2}{2 (e^{(2.1)^{-1}} - 1)(2.1)^2 d}} \le \exp\naw{-\frac{t^2}{6 d}},
\end{split}
\]
since $\E\naw{1-W_i^2}^2 = \E W_i^4 - 2\E W_i^2 + 1 = \E W_i^4 - 1 \le 2.1$ by Lemma~\ref{lemma:a-few-first-moments}.
\end{proof}
\begin{lemma}
  \label{lemma:lower-tail}
  Let $Y_1, \ldots, Y_n$ are independent random variables, $\E Y_i = 0$, $\E Y_i^2 \le \si^2$ and $Y_i \le 1$ a.s. Then for any $t > 0$,
\[
  \Pr(Y_1 + \ldots + Y_n \ge t) \le \exp\naw{- \frac{t^2}{2\naw{e^{\si^{-2}} - 1} \si^4 n}}.
\]
\end{lemma}
\begin{proof}
  Clearly, we may assume $t \le n$. Let $a > 0$ to be specified later. Using the inequality
\[
  e^x \le 1 + x + \frac{e^a - 1}{a} x^2/2,
\]
which is valid for all $x \le a$, we bound the Laplace transform of $Y_i$: for any $\lambda \le a$,
\[ \begin{split}
  \E \exp(\lambda Y_i) &\le \E \naw{1 + \lambda Y_i + \frac{e^a - 1}{a} \lambda^2 Y_i^2/2} \le 1 + \frac{e^a - 1}{a} \lambda^2 \si^2/2 \\[1ex]
  &\le \exp\naw{ \frac{e^a - 1}{a} \lambda^2 \si^2/2 }.
\end{split} \]
Therefore, for any $\lambda \le a$,
\begin{equation}
  \label{ineq:upper-tail-of-bounded-rv}
  \Pr(Y_1 + \ldots + Y_n \ge t) \le \exp\naw{-\lambda t + \frac{e^a - 1}{a} \lambda^2 n \si^2 / 2}.
\end{equation}
Taking $\lambda = \frac{t}{\frac{e^a - 1}{a} n \si^2}$ and $a = \si^{-2}$ we clearly have $\lambda \le a$ thus~(\ref{ineq:upper-tail-of-bounded-rv}) finishes the proof.
\end{proof}

\begin{proof}[Proof of Theorem~\ref{thm:l2l2}]
  Fix a unit vector $u \in \ell_2^n$.
  Let $V = H D u$ and set $\A := \frac{\sqrt{2e \log(2n/\delta)}}{\sqrt{n}}$.
  By Proposition~\ref{prop:hadamard-small-ell-infty},
\begin{equation}\label{ineq:V-in-A}
\Pr\naw{\norm{V}_\infty \le \A} \ge 1-\delta.
\end{equation}
Also, recall $\norm{V}_2 = 1$ a.s.

Now, assume $v \in \ell_2^n$ is a fixed unit vector satisfying $\norm{v}_\infty \le \A$. We shall show that
\begin{equation}
  \label{ineq:dev-for-l2}
  \Pr\naw{\frac{1}{\sqrt{d}} \norm{P v}_2 \ge 1+\eps} \le 2\delta/3,
  \qquad \text{and} \qquad
  \Pr\naw{\frac{1}{\sqrt{d}} \norm{P v}_2 \le \frac{1}{1+\eps}} \le \delta/3.
\end{equation}
Since $k \ge 20e \log(2n/\delta)$, (\ref{ineq:nka}) is satisfied. For the first inequality in~(\ref{ineq:dev-for-l2}), Proposition~\ref{prop:deviation-for-sum-of-squares} implies
\[ \begin{split}
  \Pr\naw{\frac{1}{\sqrt{d}} \norm{P v}_2 \ge 1+\eps} &\le
  \Pr\naw{\norm{P v}_2^2 \ge d + 2d \eps} \\[1ex]
  &\le \exp \naw{ - \frac{4d \eps^2}{6.2 + 24 \eps}} + \exp\naw{-\frac{3k}{8e \log(2n/\delta)} + \log(2d)}.
\end{split} \]
By an elementary calculation one can check that for $d$ and $k$ satisfying the assumptions of the theorem, both $\exp$ terms do not exceed $\delta/3$. For the second inequality in~(\ref{ineq:dev-for-l2}), Proposition~\ref{prop:deviation-for-sum-of-squares} gives
\[
  \Pr\naw{\frac{1}{\sqrt{d}} \norm{P v}_2 \le \frac{1}{1+\eps}} \le
  \Pr\naw{\norm{P v}_2^2 \le d \naw{1- \frac{2\eps}{1+2\eps}}}
  \le \exp \naw{ -\frac{2 d \eps^2}{3(1+2\eps)^2}},
\]
where the last $\exp$ term is $\le \delta/3$ for $d$ satisfying the hypothesis of the theorem. Finally, since the matrix $P$ and the vector $V = HD u$ are independent, conditioning on $V$ and combining (\ref{ineq:V-in-A}) with (\ref{ineq:dev-for-l2}) complete the proof.
\end{proof}

\subsubsection{The $\ell_1$ case ($q=1$)}

\begin{prop}
  \label{prop:deviation-for-sum-of-abs}
  Assume $v \in \ell_2^n$, $\norm{v}_2 = 1$, $\norm{v}_\infty \le \alpha$ and let $W = P v$. If (\ref{ineq:nka}) holds, then for any $t \ge 0$,
\[
  \Pr\naw{\bigabs{ |W_1| +\ldots + |W_d| - (\E |W_1| + \ldots + \E |W_d|) } \ge t} \le 2 \exp\naw{- \frac{t^2}{2d + 8t/3}}.
\]
\end{prop}
\begin{proof}
  Denote $Z = |W_1| + \ldots + |W_d|$.
  To estimate $\Pr(|Z-\E Z| \ge t)$ we use Bernstein inequality (Theorem~\ref{theorem:classical-bernstein}). To this
  end, we bound the moments of $|W_i|$. By Proposition~\ref{prop:tail-of-Wi} and integration by parts, for any integer $p \ge 1$,
  \[ \begin{split}
    \E |W_i|^p &= p \int_0^\infty t^{p-1} \Pr(|W_i| > t) \, dt \\[1ex]
                  &\le 2p \int_0^\infty t^{p-1} e^{-t^2/4} \, dt
                     + 2p \int_0^\infty t^{p-1} \exp\naw{-\frac{3t}{4 (n/k)^{1/2} \A}} \, ds \\[1ex]
                  &= p 2^p \Gamma(p/2) + 2p! \naw{\frac43 (n/k)^{1/2} \A}^p.
  \end{split} \]
Plugging the condition~(\ref{ineq:nka}) we obtain
  \begin{equation*}
    \E |W_i|^p \le p \Gamma(p/2) 2^p + 2p! \naw{\frac{4}{3\sqrt{10}}}^p.
  \end{equation*}
  However, for $p \le 6$ we can provide better bounds. For $p = 4, 6$ we can use Lemma~\ref{lemma:a-few-first-moments} and for $p=3, 5$ we use the Cauchy-Schwarz inequality $\E |W_i|^p \le \sqrt{\E |W_i|^{p-1} \E |W_i|^{p+1}}$.

Next, Lemma~\ref{lemma:central-moments} and~Proposition~\ref{prop:normal-approx} combined with~(\ref{ineq:nka}) imply
  \[ \begin{split}
    \E \bigabs{|W_i| - \E |W_i|}^p &\le \E |W_i|^p + \naw{\E|W_i|}^p \\[1ex]
      &\le \E |W_i|^p + \naw{\sqrt{2/\pi} + \frac{3}{2\sqrt{10}}}^p \le \E |W_i|^p + (4/3)^p.
  \end{split} \]
Also by Proposition~\ref{prop:normal-approx} and~(\ref{ineq:nka}),
  \[ \begin{split}
    \E\bigabs{|W_i|-\E |W_i|}^2 &= \E |W_i|^2 - (\E |W_i|)^2 \\[1ex]
      &\le 1 - \naw{\sqrt{2/\pi} - \frac{3}{2\sqrt{10}}}^2 \le \frac{9}{10}.
  \end{split} \]
All these yield the bound
\begin{equation*}
  \E\bigabs{|W_i|-\E |W_i|}^p \le \frac{p!}{2} \si_i^2 M^{p-2}
\end{equation*}
for any integer $p \ge 2$ with $\si_i^2 = 1$ and $M = 4/3$, as illustrated by the following table:
\begin{center}
  \begin{tabular}{| r | r | r |}
     \hline
     $p$ & $\E \bigabs{|W_i| - \E |W_i|}^p$ & $\frac{p!}{2} \si_i^2 M^{p-2}$ \\
\hline
    2 & $\le 0.9$ & 1 \\
    3 & $\le 3.83$ & 4 \\
    4 & $\le 5.73$ & $\approx 21$ \\
    5 & $\le 10.6$ & $\approx 142$ \\
    6 & $\le 21.24$ & $\approx 1138$ \\ \hline
   $\ge 7$ & $\le p \Gamma(p/2) 2^p + 2p! \naw{\frac{4}{3\sqrt{10}}}^p + (4/3)^p$ &
$\frac{p!}{2} (4/3)^{p-2}$ \\ \hline
\end{tabular}
\end{center}
Now, Bernstein's inequality (Theorem~\ref{theorem:classical-bernstein}) completes the proof.
\end{proof}

\begin{proof}[Proof of Theorem~\ref{thm:l2l1}]
As in the proof of Theorem~\ref{thm:l2l2}, it is enough to show
\begin{equation}
  \label{ineq:dev-for-l1}
  \Pr\naw{\abs{\frac{1}{d} \norm{P v}_1 - \sqrt{2/\pi}} \ge \eps \sqrt{2/\pi}} \le \delta
\end{equation}
for any unit $v \in \ell_2^n$ satisfying $\norm{v}_\infty \le \A := \sqrt{2e \log(2n/\delta)}/\sqrt{n}$.

Fix $\kappa \in (0,1)$. The condition~(\ref{ineq:nka}) is satisfied since $k \ge 20 e \log(2n/\delta)$. Proposition~\ref{prop:deviation-for-sum-of-abs} used for $t = d \kappa \eps \sqrt{2/\pi}$ implies
\[
  \Pr\naw{\frac{1}{d}\bigabs{ \norm{P v}_1 - \E \norm{P v}_1} \ge \kappa \eps \sqrt{2/\pi}} \le 2 \exp\naw{- \frac{\kappa^2 (2/\pi) d \eps^2}{2 + \sqrt{2/\pi}\frac83 \kappa \eps}} \le
  \delta,
\]
where the last inequality follows from the assumption on $d$, while
Proposition~\ref{prop:normal-approx} yields
  \[ 
    \abs{\frac{1}{d} \E \norm{P_v}_1 - \sqrt{2/\pi}} 
\le \frac{3 \sqrt{2e \log(2n/\delta)}}{2 \sqrt{k}} \le (1-\kappa) \eps \sqrt{2/\pi},
\]
where the last inequality follows from the assumption on $k$.
\end{proof}

\paragraph{Acknowledgments:} A preliminary version of this work was presented at the conference ``Perspectives in High Dimensions'', held at Case Western Reserve University in August 2010. The author would like to thank Prof. Louis H. Y. Chen for a discussion on normal approximation problem encountered in this work
(Theorem~\ref{theorem:stein-approximation}).


\begin{thebibliography}{10}

\bibitem{achlioptas}
D. Achlioptas, Database-friendly random projections: Johnson-Lindenstrauss with binary coins, J Comput System Sci, 66 (2003), 671--687.

\bibitem{ailon-chazelle}
N. Ailon and B. Chazelle, Approximate nearest neighbors and the fast Johnson-Lindenstrauss transform,
   Proc 38th ACM Symp Theory of Computing, 2006, pp. 557--563.

\bibitem{ailon-liberty}
N. Ailon and E. Liberty, Fast dimension reduction using Rademacher series on dual BCH codes. Discrete Comput Geom 42
   (2009), no. 4, 615--630.

\bibitem{alon}
N. Alon, L. Babai, and A. Itai, A fast and simple randomized parallel algorithm for the maximal independent set problem, J Algorithms 7 (1986), no. 4, 567--583.

\bibitem{bolthausen}
E. Bolthausen, An estimate of the remainder in a combinatorial central limit theorem,
Z Wahrsch Verw Gebiete 66 (1984), no. 3, 379--386.

\bibitem{bonami}
A. Bonami, {\'E}tude des coefficients de Fourier des fonctions de $L^p(G)$, Ann Inst Fourier (Grenoble) 20 (1970), 335--402.

\bibitem{intro-to-stein-method}
L. H. Y. Chen and Q. M. Shao,
Stein's method for normal approximation, in: An introduction to Stein's method,
Lect Notes Ser Inst Math Sci Natl Univ Singap, 4, Singapore Univ Press, Singapore, 2005, pp. 1--59.


\bibitem{chen-fang}
L. H. Y. Chen and X. Fang, On the error bound in a combinatorial central limit theorem, preprint, \texttt{arXiv:1111.3159v1}


\bibitem{dasgupta-gupta}
S. Dasgupta and A. Gupta, An elementary proof of a theorem of Johnson and Lindenstrauss, Random Structures Algorithms 22 (2003), 60--65.



\bibitem{hinrichs-vybiral}
A. Hinrichs and J. Vyb{\'i}ral, Johnson-Lindenstrauss lemma for circulant matrices, Random Structures Algorithms 39 (2011), 391--398.

\bibitem{ho-chen}
S. T. Ho and L. H. Y. Chen, An $L_p$ bound for the remainder in a combinatorial central limit theorem, Ann Probab 6 (1978), no. 2, 231--249.

\bibitem{hoeffding}
W. Hoeffding, Probability inequalities for sums of bounded random variables,
J Amer Statist Assoc 58 (1963), 13--30.

\bibitem{indyk-motwani}
P. Indyk and R. Motwani, Approximate nearest neighbors: Towards removing the curse of
dimensionality, Proc 30th Annu ACM Symp Theory of Computing, Dallas, TX, 1998, pp.
604--613.

\bibitem{joag-dev-proschan}
K. Joag-Dev and F. Proschan, Negative association of random variables, with applications, Ann Statist 11 (1983), 286--295.

\bibitem{johnson-lindenstrauss}
W. B. Johnson and J. Lindenstrauss, Extensions of Lipschitz mappings into a Hilbert space,
Contemp Math 26 (1984), 189--206.


\bibitem{kane-meka-nelson} D. M. Kane, R. Meka, and J. Nelson, Almost optimal explicit Johnson-Lindenstrauss transformations, Proc 15th Int Workshop on Randomization and Computation 2011.

\bibitem{matousek}
J. Matou\v{s}ek, On variants of the Johnson-Lindenstrauss lemma, Random Structures Algorithms 33 (2008), 142--156.

\bibitem{MS}
F. J. MacWilliams and N. J. A. Sloane, The theory of error-correcting codes, North-Holland Mathematical Library, Vol. 16. North-Holland Publishing Co., 1977.

\bibitem{shao}
Q. M. Shao, A comparison theorem on moment inequalities between negatively associated and independent random variables, J Theoret Probab 13 (2000), 343--356.

\bibitem{szarek-werner}
S. J. Szarek and E. Werner, A nonsymmetric correlation inequality for Gaussian measure, J Multivariate Anal 68 (1999), no. 2, 193--211.

\bibitem{vybiral}
J. Vyb{\'i}ral, A variant of the Johnson-Lindenstrauss lemma for circulant matrices, J Funct Anal 260 (2011), 1096--1105.

\end{thebibliography}
\end{document}